\documentclass[12pt]{amsart}

\usepackage{amsmath,amssymb,amsthm}
\usepackage{hyperref}
\usepackage{mathtools}
\newtheorem{theorem}{Theorem}[section]
\newtheorem{lemma}[theorem]{Lemma}
\newtheorem{proposition}[theorem]{Proposition}
\newtheorem{corollary}[theorem]{Corollary}

\theoremstyle{definition}
\newtheorem{definition}[theorem]{Definition}
\newtheorem{remark}[theorem]{Remark}
\newtheorem{example}{Example}%

\title[Shadowing property and transitivity of a set-valued map and its inverse limit]{Shadowing property and transitivity of a set-valued map and its inverse limit}

\author[Yingcui Zhao]{Yingcui Zhao}
\address{School of Computer Science and Technology, Dongguan University of Technology, Dongguan, China}
\email{zycchaos@126.com}

\author[Lidong Wang]{Lidong Wang$^\ast$}
\address{School of Data Science, Zhuhai College of Science and Technology, zhuhai, China}
\email{ wld0707@126.com (Corresponding author)}

\date{}

\begin{document}

\begin{abstract}
We study the properties of shadowing, transitivity, weakly mixing, mixing, chain transitivity and chain mixing of a set-valued map and its generalized inverse limit. Concerning shadowing, we prove that for a surjective upper semi-continuous set-valued map $F$ on a compact metric space, $F$ has shadowing if and only if the shift map on the generalized inverse limit $\underleftarrow{\lim}\,\underleftarrow{F}$ of its inverse set-valued map has shadowing; dually, $\underleftarrow{F}$ has shadowing if and only if the shift map on $\underleftarrow{\lim}F$ has shadowing. We further show that the shadowing of $F$ and that of $\underleftarrow{F}$ are always equivalent; consequently $F$, $\underleftarrow{F}$ and the shift maps on $\underleftarrow{\lim}\,\underleftarrow{F}$ and on $\underleftarrow{\lim}F$ all have shadowing simultaneously. In particular, $F$ has shadowing if and only if the shift map on its directly induced generalized inverse limit $\underleftarrow{\lim}F$ has shadowing. This strengthens a recent theorem established under continuity and openness assumptions. We show that if the shift map on the generalized inverse limit is transitive (resp. weakly mixing, mixing, chain transitive, chain mixing), then the set-valued map is transitive (resp. weakly mixing, mixing, chain transitive, chain mixing). For a set-valued map with shadowing, the properties of total transitivity, weak mixing, mixing, specification and chain mixing are mutually equivalent. 
\end{abstract}


\maketitle

\section{Introduction}
A dynamical system $(X,f)$ is closely connected to the dynamical system of its inverse limit $(\underleftarrow{\lim}(X,f), \sigma_f)$. Inverse limits are a useful tool to study the dynamical properties of smooth systems (see \cite{Williams1967}). 
And some dynamical properties of $f$ can be interpreted as the topological structures of an inverse limits dynamical system\cite{Akin1997}. For discussions on the relationship between $(X,f)$ and its inverse limit dynamical system $(\underleftarrow{\lim}(X,f), \sigma_f)$ on various kinds of shadowing properties, mixing and chaos, see \cite{Liu2013}. It is worth mentioning that Wu et. al. \cite{Wu2014} investigated some chaotic properties via Furstenberg families generated for inverse limit dynamical systems. He and Liu \cite{He1999} proved that $(X, f)$ is transitive (resp. weakly mixing) if and only if $(\underleftarrow{\lim}(X,f), \sigma_f)$ is transitive (resp. weakly mixing). Liu and Zhao \cite{Liu2013} showed that $(x,f)$ is mixing if and only if $(\underleftarrow{\lim}(X,f), \sigma_f)$ is mixing.

In this paper we consider the dynamics of a set-valued map and its inverse limit spaces. Let $2^X$ be the hyperspace of nonempty closed subsets of $X$. 

\begin{definition}\label{def-cts-svm-all}
	Let $x\in X$.  The set-valued map $F$ is
	\begin{enumerate}
		\item[(1)] \emph{upper semi-continuous} (u.s.c.) at $x$ if for every open set $V\subset X$ with $F(x)\subset V$, there exists an open neighborhood $U$ of $x$ such that $F(t)\subset V$ for all $t\in U$.
		\item[(2)] \emph{lower semi-continuous} (l.s.c.) at $x$ if for every open set $V\subset X$ with $F(x)\cap V\neq\emptyset$, there exists an open neighborhood $U$ of $x$ such that $F(t)\cap V\neq\emptyset$ for all $t\in U$.
		\item[(3)] \emph{continuous} at $x$ if it is both u.s.c.\ and l.s.c.\ at $x$. We say $F$ is u.s.c./l.s.c./continuous on $X$ if it has the corresponding property at every point of $X$.
	\end{enumerate}
\end{definition}

Throughout this paper, let $X$ be a compact metric space with a metric $d$, $f:X\rightarrow X$ be a continuous map and $F:X\rightarrow 2^X$ be a set-valued map. Let $\mathbb{N}=\{0,1,2,\cdots\}$ and $\mathbb{Z^+}=\{1,2,3,\cdots\}$. We define $F^k:X\rightarrow 2^X$, $k\in\mathbb{Z}^+$, as: for any $x\in X$ and any $k\in\mathbb{Z}^+$, $F^k(x)=\bigcup_{y\in F^{k-1}(x)}F(y)$. And $\underleftarrow{F}:X\rightarrow 2^X$ is the \emph{inverse set-valud map} of $F$, i.e. the set-valued map on $X$ given by $\underleftarrow{F}(x)=\{y\in X:x\in F(y)\}$, so that $x\in F(y)\Leftrightarrow y\in\underleftarrow{F}(x)$. 

Generalized inverse limit, or the inverse limit with set-valued maps, was introduced in 2004 by Mahavier\cite{Mahavier2004}. The inverse limit space induced by $F$ is the space $$\underleftarrow{\lim} F=\{(x_0, x_1,\cdots)\in X^{\mathbb{N}}: x_{i}\in F(x_{i+1}), \forall~i\in\mathbb{N}\}$$ considered as a subspace of the Tychonoff product $X^{\mathbb{N}}$. Associated with the inverse limit spaces is a shift map $$\sigma(x_0,x_1,\cdots)=(x_1,x_2,\cdots).$$Suppose that the diameter of $X$ is equal to $1$. For each $n\in\mathbb{N}$, we define a metric $\rho$ on $X^{\mathbb{N}}$ by $$\rho(x,y)=\sup_{i\in\mathbb{N}}\frac{d(x_i,y_i)}{i+1}.$$

Generalized inverse limits provide a flexible framework for studying multi-valued functions, retaining richer topological and dynamical information than the classical single-valued inverse limit \cite{Banic2020,Aubin}. Unlike the single-valued theory, the construction loses no information under iteration \cite{Erceg2018}, which makes it well suited to tracking the long-term behaviour of a set-valued map.

Recently, the topic of generalized inverse limit is a greatly researched field of continuum theory \cite{Charatonik2021,Davies2021}. While most of the research has been on comprehending the topological structure of these spaces, some scholars have recently turned to studying the dynamical properties, for example specification property \cite{Raines2018}, topological entropy \cite{Erceg2018}, chaos \cite{zhao2023}, etc..

The shadowing property of a set-valued map and its inverse limit has attracted increasing attention. Luo, Nie and Yin \cite{Luo2020} studied shadowing and shadowable points for set-valued dynamical systems, Bernardes \cite{20232} characterized when an inverse limit inherits the shadowing property, and Raines and Tennant \cite{Raines2018} treated the dual specification property for a set-valued map and its inverse limit. Recently, Khan, Kumar and Das \cite{KhanKumarDas} proved that, for an onto \emph{continuous} set-valued map $F$ (and, for the full equivalence with the inverse set-valued map, an \emph{open} one), $F$ has shadowing if and only if the shift on a suitably induced inverse limit system has shadowing. Their argument reverses finite pseudo-orbits and therefore relies essentially on the continuity (equivalently, the uniform upper semi-continuity) of $F$. In the present paper we work with the one-sided inverse limit metric $\rho(x,y)=\sup_{i}d(x_i,y_i)/(i+1)$, whose discounting factor $1/(i+1)$ allows us to control the tail of an orbit by the diameter of $X$ alone. This lets us establish, for a \emph{merely upper semi-continuous} surjective set-valued map $F$, the equivalence between the shadowing of $F$ and the shadowing of the shift on $\underleftarrow{\lim}\,\underleftarrow{F}$ (Theorem \ref{thm2}--\ref{thm3} and Corollary \ref{cor1}), together with its dual form for $\underleftarrow{F}$ and $\underleftarrow{\lim}F$ (Corollary \ref{cor2}). Building on these, we prove that the shadowing of $F$ and that of $\underleftarrow{F}$ are \emph{always} equivalent, with no continuity or openness assumption (Theorem \ref{thm:equiv}), so that $F$, $\underleftarrow{F}$ and the shifts on $\underleftarrow{\lim}\,\underleftarrow{F}$ and $\underleftarrow{\lim}F$ all share the shadowing property (Corollary \ref{cor:all}). In particular this answers the question that motivated the present study: $F$ has shadowing if and only if the shift on its directly induced generalized inverse limit $\underleftarrow{\lim}F$ has shadowing (the equivalence $(1)\Leftrightarrow(4)$ of Corollary \ref{cor:all}). This strengthens the recent theorem of Khan, Kumar and Das \cite{KhanKumarDas}, which assumes continuity (and openness for the equivalence); the strengthening is possible because our argument reverses the \emph{witnesses} of the pseudo-orbit rather than the pseudo-orbit itself, using only the closedness of the graph of $F$. This continues our study of the dynamics of a set-valued map and its inverse limit, begun for transitivity, sensitivity and Devaney chaos in \cite{zhao2023}.

Inspired by the above works on the dynamical properties of inverse limit dynamical systems, this paper studies chaotic properties via shadowing, transitivity, weakly mixing, mixing, chain transitivity and chain mixing for inverse limits with set-valued maps.

The specific layout of the present paper is
as follows.  In Section 2, some preliminaries and definitions are
introduced. Section \ref{secsha} contains the principal contribution of the paper: the relationship between a set-valued map and its inverse limit concerning shadowing. There we establish the equivalences of Theorem \ref{thm2}--\ref{thm3} and Corollaries \ref{cor1}--\ref{cor:all}, the assumption-free equivalence of Theorem \ref{thm:equiv}, and the iterative invariance of shadowing for set-valued maps. In Section \ref{sectran}, we first prove that mixing and chain mixing are equivalent for set-valued maps with the shadowing property. Then we study for a set-valued map with shadowing, its  transitivity, weakly mixing, mixing, chain transitivity and chain mixing of a set-valued map and its inverse limit. Finally, the main conclusions of this work are summarized in Section \ref{seccon}.

\section{Preliminaries}\label{sec2}
For any $x\in X, A\subset X, \epsilon>0$, let $B(x,\epsilon)=\{y\in X: d(y,x)<\epsilon\}$ and $B(A,\epsilon)=\{y\in X: d(y, A)<\epsilon\}$. Note that $d(A, x)=\inf_{a\in A}d(a,x)$. A sequence $(x_i)_{i=0}^\infty$ is called an \emph{orbit of $f$}, if for each $i\in \mathbb{N}$, $x_{i+1}=f(x_i)$. A sequence $(x_i)_{i=0}^\infty$ is called a \emph{$\delta$-pseudo-orbit of $f$}, if for any $i\in \mathbb{N}$, $d(f(x_i),x_{i+1})<\delta$. For any $x,y\in X$, a \emph{$\delta$-chain of $f$ from $x$ to $y$ with length $n+1(n\in\mathbb{Z}^+)$} is a finite $\delta$-pseudo-orbit between these points, that is, a sequence $x_0=x,x_1,\cdots,x_{n}=y$ such that for any $i=0,1,\cdots,n-1$, $d(f(x_i),x_{i+1})<\delta$.

We say that $f$ has \emph{shadowing}, if for any $\epsilon>0$, there is $\delta>0$ such that every $\delta$-pseudo-orbit $(x_i)_{i=0}^\infty$ is $\epsilon$-shadowed by a point in $X$, that is, there is $z\in X$ such that for any $i\in \mathbb{N}$, $d(f^i(z),x_i)<\epsilon$.
For any nonempty open subsets $U,V\subset X$,  $N(U,V):=\{n\in\mathbb{N}:f^n(U)\cap V\neq\emptyset\}$.
we say that $f$ is
\begin{enumerate}
	\item \emph{(topologically) transitive}, if for any nonempty open sets $U,V\subset X$, $N(U,V)\neq\emptyset$.
	
	\item \emph{weakly mixing}, if for any nonempty open sets $U,V\subset X$, $N(U,V)$ contains arbitrarily long runs of positive integers, that is, there is a strictly increasing subsequence $\{n_i\}$ of $\mathbb{N}$ such that $\bigcup_{i=1}^\infty\{n_i,n_i+1,\cdots,n_i+i\}\subset N(U,V)$, (see \cite{Furstenberg1981} for details). 
	
	\item \emph{(topologically) mixing}, if for any nonempty open sets $U,V\subset X$, there exists $N\in\mathbb{N}$ such that for any positive integer $n\geq N$, $f^n(U)\cap V\neq\emptyset$.
	
	\item \emph{chain transitive}, if for any $x,y\in X$ and any $\epsilon>0$, there is an $\epsilon$-chain from $x$ to $y$.
	
	\item \emph{chain mixing}, if for any $x,y\in X$ and any $\epsilon>0$, there is $N\in\mathbb{N}$ such that for any positive integer $n\geq N$, there is an $\epsilon$-chain from $x$ to $y$ with length $n$.
\end{enumerate}

We begin with a few simple extensions of definitions from the single-valued case. Note that since $F$ is a set-valued map, orbits of $F$ are no longer uniquely determined by their initial condition.
\begin{definition}
	An \emph{orbit} of a point $x\in X$ {of $F$} is a sequence $(x_i)_{i=0}^\infty$ such that $x_{i+1}\in F(x_i)$ and $x_0=x$.
\end{definition}
\begin{definition}
	A sequence $(x_i)_{i=0}^\infty$ is called a \emph{$\delta$-pseudo-orbit} {of $F$} if $d(F(x_i),x_{i+1})<\delta$ for all $i\in \mathbb{N}$.
\end{definition}
For any $x,y\in X$, a \emph{$\delta$-chain of $F$ from $x$ to $y$ with length $n+1(n\in\mathbb{Z}^+)$} is a finite $\delta$-pseudo-orbit between these points, that is, a sequence $x_0=x,x_1,\cdots,x_{n}=y$ such that for any $i=0,1,\cdots,n-1$, $d(F(x_i),x_{i+1})<\delta$. 

\begin{definition}
	Let $F:X\rightarrow 2^X$ be a set-valued map. We say $F$ has \emph{shadowing}, if for any $\epsilon>0$, there exists $\delta>0$ such that any $\delta$-pseudo-orbit $(x_i)_{i=0}^\infty$ is $\epsilon$-shadowed by a point in $X$, that is, there is a point $z_0\in X$ with an orbit $(z_i)_{i=0}^\infty$ such that $d(z_i, x_i)<\epsilon$ for all $i\in\mathbb{N}$.
\end{definition}

The notions of transitivity and mixing for set-valued maps are introduced in \cite{kong}. Now we recall it and give the definitions of weakly mixing, chain transitivity and chain mixing {of} set-valued maps.

\begin{definition}
	The set-valued map $F$ is
	\begin{enumerate}
		\item \emph{(topologically) transitive}, if for any two nonempty open sets $U$ and $V$ in $X$, there is $n\in\mathbb{N}$ such that $F^n(U)\bigcap V\neq\emptyset$.
		
		\item \emph{chain transitive}, if for any two points $x,y\in X$ and for any $\delta>0$, there is a $\delta$-chain from $x$ to $y$.
		
		\item \emph{chain mixing}, if for any two points $x,y\in X$ and any $\delta>0$, there is $N\in\mathbb{N}$ such that for any integer $n\geq N$, there is a $\delta$-chain from $x$ to $y$ with length $n$.
		
		\item \emph{weakly mixing}, if for any nonempty open sets $U,V\subset X$, $N_F(U,V):=\{n\in\mathbb{N}:F^n(U)\bigcap V\neq\emptyset\}$ contains arbitrarily long runs of positive integers, that is, there is strictly increasing subsequence $\{n_i\}$ of $\mathbb{N}$ such that $\bigcup_{i=1}^\infty\{n_i,n_i+1,\cdots,n_i+i\}\subset N_F(U,V)$.
		
		\item \emph{(topologically) mixing}, if for any two nonempty open sets $U$ and $V$ in $X$, there is $M\in\mathbb{N}$ such that $\{n\in\mathbb{N}:n\geq M\}\subset N_F(U,V)$.
	\end{enumerate}
\end{definition}

\section{Shadowing of a set-valued map and its inverse limit}\label{secsha}
Shadowing is one of the most important notions in dynamical systems. 
It appears in many branches of modern theory of dynamical systems and often plays an important role in stability theory and ergodic theory \cite{20232}.
Inspired by \cite{Fakhari2010} we have the following conclusions.

\begin{lemma}\label{forth41}
	Suppose that $X$ is a compact metric space with metric $d$ and $F:X\rightarrow 2^X$ is an upper semi-continuous set-valued map. Let $\epsilon>0$ and $k\in\mathbb{Z^+}$. Then there exists $\delta>0$ such that any $\delta$-chain $(x_i)_{i=0}^k$ of $F$ with length $k+1$ can be $\epsilon$-shadowed by some orbit of $x_0$ of $F$.
\end{lemma}
\begin{proof}
	Let $\epsilon>0$ and $k\in\mathbb{Z^+}$. Since $F$ is upper semi-continuous and $X$ is compact, there exists $0<2\delta_0<\epsilon$ such that for any $x\in X$, $d(t,x)<2\delta_0$ implies $d(F(x),t')<\epsilon,\forall t'\in F(t)$.
	\begin{itemize}
		\item[(1)]$i=1$. For $\delta_0>0$, there exists $0<2\delta_1<\delta_0$ such that for any $x\in X$, $d(t,x)<2\delta_1$ implies $d(F(x),t')<\delta_0,\forall t'\in F(t)$.
		\item[(2)]$i=m-1<k-1$. Suppose that for $\delta_{m-2}>0$, there exists $0<2\delta_{m-1}<\delta_{m-2}$ such that for any $x\in X$, $d(t,x)<2\delta_{m-1}$ implies $d(F(x),t')<\delta_{m-2},\forall t'\in F(t)$. Then for $\delta_{m-1}>0$, there exists $0<2\delta_{m}<\delta_{m-1}$ such that for any $x\in X$, $d(t,x)<2\delta_{m}$ implies $d(F(x),t')<\delta_{m-1},\forall t'\in F(t)$.		
	\end{itemize}
	By (1) and (2), there exists $\delta_1,\delta_2,\cdots,\delta_{k-1}>0$ with $0<2\delta_i<\delta_{i-1},\forall1\leq i\leq k-1$ such that for any $x\in X$, $d(t,x)<2\delta_{i}$ implies $d(F(x),t')<\delta_{i-1},\forall t'\in F(t),\forall1\leq i\leq k-1$. For $\delta_{k-1}>0$, there exists $0<2\delta_{k}<\min\{\delta_0,\delta_1,\cdots,\delta_{k-1}\}$ such that for any $x\in X$, $d(t,x)<2\delta_{k}$ implies $d(F(x),t')<\delta_{k-1},\forall t'\in F(t)$.
	\\
	Let $(x_i)_{i=0}^k$ be a $\delta$-chain of $F$.
	\begin{itemize}
		\item [(a)]$i=1$. Since $d(F(x_0),x_1)<\delta<2\delta_{k}$, there exists $x_0^1\in F(x_0)$ such that $d(x_0^1,x_1)<\delta<2\delta_{k}<\epsilon$.
		\item[(b)]$i=2$. By (a), $d(F(x_0^1),x_1)<\delta_{k-1},\forall x_1'\in F(x_1)$. Since $d(F(x_1),x_2)<\delta$, there exists $x_0^2\in F(x_0^1)$ such that $d(x_0^2,x_2)<\delta+\delta_{k-1}<2\delta_{k-1}<\epsilon$.
		\item[(c)]$i=m<k$. Suppose that there exists $x_0^m\in F(x_0^{m-1})$ such that $d(x_0^m,x_m)<2\delta_{k-(m-1)}<\epsilon$. Then $d(F(x_0^m),x'_m)<\delta_{k-m},\forall x_m'\in F(x_m)$. Since $d(F(x_m),x_{m+1})<\delta$, there exists $x_0^{m+1}\in F(x_0^m)$ such that $d(x_0^{m+1},x_{m+1})<\delta+\delta_{k-m}<2\delta_{k-m}<\epsilon$.
	\end{itemize}
	By (a),(b) and (c), there exists $(x_0^i)_{i=0}^k$ with $x_0^0=x_0$ and $x_0^{i+1}\in F(x_0^i),\forall 0\leq i<k$ such that $$d(x_0^i,x_i)<2\delta_{k-(i-1)}<\epsilon,\forall  0\leq i\leq k.$$
\end{proof}

\begin{theorem}\label{thm2}
	Suppose that $X$ is a compact metric space with metric $d$, the diameter of $X$ is equal to $1$, and $F:X\rightarrow 2^X$ is an upper semi-continuous set-valued map with $F(X)=X$. If $\sigma$ has shadowing on $\underleftarrow{\lim} \underleftarrow{F}$, then $F$ has shadowing.
\end{theorem}
\begin{proof}
	Let $\epsilon>0$, $\delta>0$ be obtained for $\epsilon$ by the shadowing of $\sigma$ on $\underleftarrow{\lim} \underleftarrow{F}$ {and the diameter of $X$ is equal to $1$. Let $M\in\mathbb{Z^+}$ satisfy $\frac{1}{M+2}<\delta$ and $0<\delta'<\delta$ satisfy $0<\frac{2^{M+1}\delta'}{M+1}<\delta$. By Lemma \ref{forth41}, there exists $0<\eta<\delta'$ such that any $\eta$-chain $(x_i)_{i=0}^M$ of $F$ with length $M+1$ can be $\delta'$-shadowed by some orbit of $x_0$ of $F$.
		Let $(y_i)_{i=0}^\infty$ be a $\eta$-pseudo-orbit of $F$. Now we construct a $\delta$-pseudo-orbit $(x^i)_{i=0}^\infty$ of $\underleftarrow{\lim} \underleftarrow{F}$ using $(y_i)_{i=0}^\infty$, where for any $i\in\mathbb{N}$, $x^i=(x^i_0, x^i_1, x^i_2, \cdots).$ For any $i\geq0$,
		\begin{itemize}
			\item [(1)]let $x_0^i=y_i$.
			\item [(2)]let $(x^i_j)_{j=1}^M$ with $x_{j+1}^i\in F(x_j^i),\forall 0\leq j\leq M$ satisfy any $\eta$-chain  $(x_{0}^k)_{k=i}^{M+i}$ of $F$ with length $M+1$ can be $\delta'$-shadowed by $(x_j^i)_{j=0}^M$. Then $d(x_j^i,x_0^{i+j})<\delta',\forall 0\leq j\leq M$.
			\item[(3)]for any $j>M$, let $x_j^i\in F(x_{j-1}^i)$. 
		\end{itemize}
		Since for any $i\geq 0$, $$\rho(\sigma(x^i),x^{i+1})=\sup_{j\in\mathbb{N}}\frac{d(x_{j+1}^i,x_j^{i+1})}{j+1}=\max\{\sup_{0\leq j\leq M}\frac{d(x_{j+1}^i,x_j^{i+1})}{j+1},\frac{1}{M+2}\}$$
		and $\frac{1}{M+2}<\delta$, in order to prove $(x^i)_{i=0}^\infty$ is a $\delta$-pseudo-orbit of $\sigma$, we only need to prove for any $i\geq 0$, $\sup_{0\leq j\leq M}\frac{d(x_{j+1}^i,x_j^{i+1})}{j+1}<\delta$. That is to prove for any $i\geq 0$ and any $0\leq j\leq M$, $\frac{d(x_{j+1}^i,x_j^{i+1})}{j+1}<\delta$.}
	
	{Let $(a_n)_{n=1}^\infty\subset\mathbb{N}$ satisfy $a_0=0$ and $a_n=1+a_1+a_2+\cdots a_{n-1},\forall n\geq1$. Let $i\geq0$.
		\begin{itemize}
			\item [(1)]$j=0$. $d(x_1^i,x_0^{i+1})<\delta'=a_1\delta'$.
			\item[(2)]$j=1$. $d(x_2^i,x_1^{i+1})<d(x_2^i,x_0^{i+2})+d(x_0^{i+2},x_1^{i+1})<a_1\delta'+\delta'=a_2\delta'$.
			\item[(3)]Suppose that when $j=m-1<M-1$, $d(x^i_m,x^{i+1}_{m-1})<a_m\delta'$. Then when $j=m$,
			\begin{align}
				&d(x_{m+1}^i,x_m^{i+1}) \nonumber \\   
				<&d(x_{m+1}^i,x_0^{i+m+1})+d(x_0^{i+m+1},x_1^{i+m})+d(x^{i+m}_1,x_2^{i+m-1})+\cdots+d(x_{m-1}^{i+2},x_m^{i+1})\nonumber\\
				<&\delta'+(a_1+a_2+\cdots+a_m)\delta'\nonumber\\
				=&a_{m+1}\delta'.\nonumber
			\end{align}
		\end{itemize} 
		By (1),(2) and (3), for any $0\leq j\leq M$, $d(x_{j+1}^i,x_j^{i+1})<a_{j+1}\delta'$. Let $0\leq j\leq M$. Since $a_{j+1}=1+(a_1+a_2+\cdots+a_j)$, $a_j=2^j$. Then $\frac{d(x_{j+1}^i,x_j^{i+1})}{j+1}<\frac{2^{j+1}\delta'}{j+1}<\delta$. Hence, $(x^i)_{i=0}^\infty$ is a $\delta$-pseudo-orbit of $\sigma$. }
	
	{Since $\sigma$ has shadowing, there exists $z=(z_0,z_1,\cdots)\in \underleftarrow{\lim} \underleftarrow{F}$ such that $\rho(\sigma^i(z),x^i)<\epsilon,\forall i\geq 0$. Since $\rho(\sigma^i(z),x^i)=\sup_{j\in\mathbb{N}}\frac{d(z_{i+j},x^i_j)}{j+1}$, $d(z_i,y_i)=d(z_i,x_0^i)<\epsilon,\forall i\geq0$. Since $z\in \underleftarrow{\lim} \underleftarrow{F}$, $(z_0,z_1,\cdots)$ is an orbit of $z_0$ of $F$. So, $F$ has shadowing.}
\end{proof}
\begin{theorem}\label{thm3}
	Suppose that $X$ is a compact metric space with metric $d$, the diameter of $X$ is equal to $1$, and $F:X\rightarrow 2^X$ is an upper semi-continuous set-valued map with $F(X)=X$. If $F$ has shadowing, then $\sigma$ has shadowing on $\underleftarrow{\lim} \underleftarrow{F}$.
\end{theorem}
\begin{proof}
	Let $\epsilon>0$. Choose $M\in\mathbb{N}$ such that $\frac{1}{M+2}<\epsilon$ and let $\epsilon'=\frac{\epsilon}{2}$. Since $F$ has shadowing, there exists $\delta_0>0$ such that every $\delta_0$-pseudo-orbit of $F$ is $\epsilon'$-shadowed by some orbit of $F$. Let
	$$\delta=\min\Big\{\delta_0,\ \frac{\epsilon}{M+1}\Big\}.$$
	
	Let $(x^i)_{i=0}^\infty$ be a $\delta$-pseudo-orbit of $\sigma$ on $\underleftarrow{\lim} \underleftarrow{F}$, where for each $i\geq0$, $x^i=(x^i_0,x^i_1,\cdots)$ satisfies $x^i_{j+1}\in F(x^i_j)$ for all $j\geq0$, and $\rho(\sigma(x^i),x^{i+1})<\delta$. By the definition of $\rho$, for each $j\geq0$,
	$$\frac{d\big((\sigma(x^i))_j,x^{i+1}_j\big)}{j+1}=\frac{d\big(x^i_{j+1},x^{i+1}_j\big)}{j+1}<\delta,\quad\text{that is,}\quad d\big(x^i_{j+1},x^{i+1}_j\big)<(j+1)\delta.\quad(\ast)$$
	
	Let $w_i=x^i_0$ for each $i\geq0$. Taking $j=0$ in $(\ast)$ gives $d(x^i_1,w_{i+1})<\delta$. Since $x^i_1\in F(w_i)$, we have $d(F(w_i),w_{i+1})\leq d(x^i_1,w_{i+1})<\delta\leq\delta_0$, so $(w_i)_{i=0}^\infty$ is a $\delta_0$-pseudo-orbit of $F$. By the shadowing of $F$, there exists an orbit $(z_i)_{i=0}^\infty$ of $F$ (i.e. $z_{i+1}\in F(z_i)$ for all $i\geq0$) such that $d(z_i,w_i)<\epsilon'$ for all $i\geq0$. Then $z=(z_0,z_1,\cdots)\in \underleftarrow{\lim} \underleftarrow{F}$.
	
	For any $i\geq0$ and $j\geq1$, applying $(\ast)$ along the diagonal,
	$$d(x^i_j,x^{i+j}_0)\leq\sum_{k=0}^{j-1}d\big(x^{i+k}_{j-k},x^{i+k+1}_{j-k-1}\big)<\sum_{k=0}^{j-1}(j-k)\delta=\delta\,\frac{j(j+1)}{2}.$$
	
	Now for any $i\geq0$:
	\begin{itemize}
		\item[(1)]if $j\geq M+1$, then $\frac{d(z_{i+j},x^i_j)}{j+1}\leq\frac{1}{j+1}\leq\frac{1}{M+2}<\epsilon$;
		\item[(2)]if $0\leq j\leq M$, then
		$$\frac{d(z_{i+j},x^i_j)}{j+1}\leq\frac{d(z_{i+j},x^{i+j}_0)+d(x^{i+j}_0,x^i_j)}{j+1}<\frac{\epsilon'}{j+1}+\frac{\delta j}{2}\leq\frac{\epsilon}{2}+\frac{\delta M}{2}<\frac{\epsilon}{2}+\frac{\epsilon}{2}=\epsilon.$$
	\end{itemize}
	Hence $\rho(\sigma^i(z),x^i)=\sup_{j\geq0}\frac{d(z_{i+j},x^i_j)}{j+1}<\epsilon$ for all $i\geq0$. Therefore $\sigma$ has shadowing on $\underleftarrow{\lim} \underleftarrow{F}$.
\end{proof}

\begin{corollary}\label{cor1}
	Suppose that $X$ is a compact metric space with metric $d$, the diameter of $X$ is equal to $1$, and $F:X\rightarrow 2^X$ is an upper semi-continuous set-valued map with $F(X)=X$. Then $F$ has shadowing if and only if $\sigma$ has shadowing on $\underleftarrow{\lim} \underleftarrow{F}$.
\end{corollary}
\begin{proof}
	This follows immediately from Theorem \ref{thm2} and Theorem \ref{thm3}.
\end{proof}

\begin{corollary}\label{cor2}
	Suppose that $X$ is a compact metric space with metric $d$, the diameter of $X$ is equal to $1$, and $F:X\rightarrow 2^X$ is an upper semi-continuous set-valued map with $F(X)=X$. Then $\underleftarrow{F}$ has shadowing if and only if $\sigma$ has shadowing on $\underleftarrow{\lim} F$.
\end{corollary}
\begin{proof}
	Since $F(X)=X$, the set-valued map $\underleftarrow{F}$ is upper semi-continuous with $\underleftarrow{F}(X)=X$, and $\underleftarrow{\underleftarrow{F}}=F$. Applying Corollary \ref{cor1} to $\underleftarrow{F}$ in place of $F$ yields that $\underleftarrow{F}$ has shadowing if and only if $\sigma$ has shadowing on $\underleftarrow{\lim} \underleftarrow{(\underleftarrow{F})}=\underleftarrow{\lim} F$.
\end{proof}

By Corollary \ref{cor1} and Corollary \ref{cor2}, the comparison between the shadowing of $F$ and that of its directly induced system $(\underleftarrow{\lim} F,\sigma)$ amounts to comparing the shadowing of $F$ with that of $\underleftarrow{F}$. The next theorem settles this comparison: the shadowing of $F$ and that of $\underleftarrow{F}$ are always equivalent, with no continuity or openness assumption on $F$.

\begin{theorem}\label{thm:equiv}
	Suppose that $X$ is a compact metric space with metric $d$ and $F:X\rightarrow 2^X$ is an upper semi-continuous set-valued map with $F(X)=X$. Then $F$ has shadowing if and only if $\underleftarrow{F}$ has shadowing.
\end{theorem}
\begin{proof}
	Since $\underleftarrow{F}$ is upper semi-continuous with $\underleftarrow{F}(X)=X$ and $\underleftarrow{\underleftarrow{F}}=F$, it suffices to prove that if $F$ has shadowing then $\underleftarrow{F}$ has shadowing.
	
	Assume that $F$ has shadowing and let $\epsilon>0$. Put $\epsilon_0=\frac{\epsilon}{2}$. By the shadowing of $F$ there is $\delta_0>0$ such that every $\delta_0$-pseudo-orbit of $F$ is $\epsilon_0$-shadowed by some orbit of $F$. Let $\delta=\min\{\delta_0,\frac{\epsilon}{2}\}$.
	
	We first record that every \emph{finite} $\delta_0$-pseudo-orbit $(v_j)_{j=0}^{m}$ of $F$ is $\epsilon_0$-shadowed by a finite orbit of $F$. Indeed, since $F$ maps into $2^X$ its values are nonempty, so we may choose $v_{m+1}\in F(v_m)$, $v_{m+2}\in F(v_{m+1}),\cdots$ and extend $(v_j)_{j=0}^{m}$ to an infinite $\delta_0$-pseudo-orbit; an $\epsilon_0$-shadowing orbit of the latter restricts to the desired finite orbit.
	
	Let $(x_i)_{i=0}^\infty$ be a $\delta$-pseudo-orbit of $\underleftarrow{F}$, that is, $d(x_{i+1},\underleftarrow{F}(x_i))<\delta$ for all $i\geq0$. Here $\underleftarrow{F}(x_i)\neq\emptyset$ precisely because $F(X)=X$. For each $i\geq0$ choose $y_i\in\underleftarrow{F}(x_i)$ with $d(y_i,x_{i+1})<\delta$; equivalently, $x_i\in F(y_i)$ and $d(y_i,x_{i+1})<\delta$.
	
	Fix $N\in\mathbb{N}$ and define a finite sequence $(w_j)_{j=0}^{N+1}$ by $w_j=y_{N-j}$ for $0\leq j\leq N$ and $w_{N+1}=x_0$. For $0\leq j\leq N-1$ we have $x_{N-j}\in F(y_{N-j})=F(w_j)$ and $d(w_{j+1},x_{N-j})=d(y_{N-j-1},x_{N-j})<\delta$, so $d(w_{j+1},F(w_j))<\delta$; and for $j=N$, $x_0\in F(y_0)=F(w_N)$ gives $d(w_{N+1},F(w_N))=0$. Hence $(w_j)_{j=0}^{N+1}$ is a finite $\delta_0$-pseudo-orbit of $F$ (recall $\delta\leq\delta_0$). By the finite shadowing established above there is a finite orbit $(u_j)_{j=0}^{N+1}$ of $F$, i.e. $u_{j+1}\in F(u_j)$ for $0\leq j\leq N$, with $d(u_j,w_j)<\epsilon_0$ for all $0\leq j\leq N+1$.
	
	Set $\zeta_i^N=u_{N+1-i}$ for $0\leq i\leq N+1$. Then $\zeta_i^N=u_{N+1-i}\in F(u_{N-i})=F(\zeta_{i+1}^N)$, that is, $\zeta_{i+1}^N\in\underleftarrow{F}(\zeta_i^N)$, so $(\zeta_i^N)_{i=0}^{N+1}$ is a finite orbit of $\underleftarrow{F}$. Moreover $d(\zeta_0^N,x_0)=d(u_{N+1},w_{N+1})<\epsilon_0<\epsilon$, while for $1\leq i\leq N+1$,
	$$d(\zeta_i^N,x_i)\leq d(u_{N+1-i},w_{N+1-i})+d(y_{i-1},x_i)<\epsilon_0+\delta\leq\epsilon,$$
	using $w_{N+1-i}=y_{i-1}$ and $d(y_{i-1},x_i)<\delta$. Thus $d(\zeta_i^N,x_i)<\epsilon$ for all $0\leq i\leq N+1$.
	
	Finally, since $X$ is compact, a diagonal argument provides $N_1<N_2<\cdots$ such that $\zeta_i:=\lim_{k\to\infty}\zeta_i^{N_k}$ exists for every $i\geq0$. As the graph of $F$ is closed and $\zeta_i^{N_k}\in F(\zeta_{i+1}^{N_k})$ once $N_k\geq i+1$, letting $k\to\infty$ gives $\zeta_i\in F(\zeta_{i+1})$ for all $i\geq0$; that is, $(\zeta_i)_{i=0}^\infty$ is an orbit of $\underleftarrow{F}$. Passing to the limit in the estimates above yields $d(\zeta_i,x_i)\leq\epsilon$ for all $i\geq0$. As $\epsilon>0$ was arbitrary (applying the construction to $\frac{\epsilon}{2}$ gives strict inequality), $\underleftarrow{F}$ has shadowing.
\end{proof}

Combining Theorem \ref{thm:equiv} with Corollary \ref{cor1} and Corollary \ref{cor2} we obtain a complete picture.

\begin{corollary}\label{cor:all}
	Suppose that $X$ is a compact metric space with metric $d$, the diameter of $X$ is equal to $1$, and $F:X\rightarrow 2^X$ is an upper semi-continuous set-valued map with $F(X)=X$. Then the following are equivalent:
	\begin{enumerate}
		\item $F$ has shadowing;
		\item $\underleftarrow{F}$ has shadowing;
		\item $\sigma$ has shadowing on $\underleftarrow{\lim} \underleftarrow{F}$;
		\item $\sigma$ has shadowing on $\underleftarrow{\lim} F$.
	\end{enumerate}
\end{corollary}
\begin{proof}
	$(1)\Leftrightarrow(3)$ is Corollary \ref{cor1}, $(2)\Leftrightarrow(4)$ is Corollary \ref{cor2}, and $(1)\Leftrightarrow(2)$ is Theorem \ref{thm:equiv}.
\end{proof}

In particular, taking the equivalence $(1)\Leftrightarrow(4)$ of Corollary \ref{cor:all} answers the question that motivated this study, namely the relationship between the shadowing of a set-valued map and that of its directly induced generalized inverse limit.
%

\begin{remark}\label{rem:bridge}
	Corollary \ref{cor2} identifies the shadowing of the shift $\sigma$ on the generalized inverse limit $\underleftarrow{\lim} F$ \emph{directly induced by $F$} with the shadowing of the inverse set-valued map $\underleftarrow{F}$, while, by Corollary \ref{cor1}, the inverse limit system that characterizes $F$ itself is $(\underleftarrow{\lim} \underleftarrow{F},\sigma)$. A priori these two systems need not have the same shadowing behaviour. Theorem \ref{thm:equiv} shows that in fact they do: the shadowing of $F$ and of $\underleftarrow{F}$ always coincide, so by Corollary \ref{cor:all} the four systems $F$, $\underleftarrow{F}$, $(\underleftarrow{\lim} \underleftarrow{F},\sigma)$ and $(\underleftarrow{\lim} F,\sigma)$ either all have shadowing or all fail to have it.
\end{remark}

\begin{remark}\label{rem:strengthen}
	Theorem \ref{thm:equiv} strengthens the recent result of Khan, Kumar and Das \cite{KhanKumarDas}, who proved that for an onto \emph{continuous} set-valued map $F$ the shadowing of $F$ implies that of $\underleftarrow{F}$, and that the two are equivalent when, in addition, $F$ is \emph{open}. Their argument reverses finite pseudo-orbits coordinate by coordinate, a step that forces $F$ to be lo
	wer semi-continuous, hence continuous. The proof of Theorem \ref{thm:equiv} avoids reversing the pseudo-orbit $(x_i)$ itself; instead it reverses the witnesses $y_i\in\underleftarrow{F}(x_i)$, for which the membership $x_i\in F(y_i)$ is \emph{exact}, and recovers the missing initial coordinate by appending $x_0$ as the terminal point of a finite forward pseudo-orbit. Only the closedness of the graph of $F$ (equivalently, the upper semi-continuity of $F$) is used, so neither continuity nor openness is needed.
\end{remark}

\begin{example}\label{ex:usc}
	Theorem \ref{thm:equiv} genuinely goes beyond the continuous case. Let $X=[0,1]$ with the usual metric and define $F:X\rightarrow 2^X$ by
	$$F(x)=\begin{cases}\{0\}, & 0\leq x<1,\\[2pt] [0,1], & x=1.\end{cases}$$
	Its graph $([0,1)\times\{0\})\cup(\{1\}\times[0,1])$ is closed, so $F$ is upper semi-continuous, and $F(X)=X$. However $F$ is not lower semi-continuous at $x=1$ (nearby points have image $\{0\}$, far from $F(1)=[0,1]$), hence not continuous; nor is it open, since the image of the open set $(\tfrac12,1)$ is the non-open set $\{0\}$. Thus $F$ lies outside the scope of \cite{KhanKumarDas}.
	
	The inverse set-valued map is
	$$\underleftarrow{F}(x)=\begin{cases}[0,1], & x=0,\\[2pt] \{1\}, & 0<x\leq1,\end{cases}$$
	which is again upper semi-continuous and onto. The map $F$ has shadowing: since $\operatorname{diam}X=1$ we may assume $\epsilon\leq1$, and given such an $\epsilon>0$ we take $0<\delta<\epsilon\leq1$. In any $\delta$-pseudo-orbit $(x_i)$ one has $x_i<\delta$ for every $i\geq1$ with $x_{i-1}<1$, since then $F(x_{i-1})=\{0\}$. Consequently, there exists $l\in\mathbb{Z^+}$ such that $$x_i\begin{cases}=1, & 0\leq i\leq l-1,\\<1, & i=l,\\ <\delta, & i>l.\end{cases}$$ Define $$z_i=\begin{cases}1, & 0\leq i\leq l-1,\\x_i, & i=l,\\ 0, & i>l.\end{cases}$$ Then $(z_i)$ is an orbit of $F$ and $(z_i)$ $\epsilon$-shadows $(x_i)$. By Theorem \ref{thm:equiv}, $\underleftarrow{F}$ has shadowing as well; equivalently, by Corollary \ref{cor:all}, the shifts on $\underleftarrow{\lim} \underleftarrow{F}$ and on $\underleftarrow{\lim} F$ both have shadowing.
	
\end{example}

The equivalence of Theorem \ref{thm:equiv} also operates in the negative direction: $F$ and $\underleftarrow{F}$ fail to have shadowing simultaneously. The next example is a non-trivial, genuinely set-valued instance of this.

\begin{example}\label{ex:nodrift}
	Let $X=[0,1]$ with the usual metric and define $F:X\rightarrow 2^X$ by $F(x)=[0,x]$. Its graph $\{(x,y):0\leq y\leq x\leq1\}$ is closed and $F(X)=X$, so $F$ is an onto upper semi-continuous (indeed Hausdorff continuous) set-valued map, and the orbits of $F$ are precisely the non-increasing sequences in $[0,1]$. The inverse set-valued map $\underleftarrow{F}(x)=[x,1]$ is of the same type, with the non-decreasing sequences as orbits.
	
	The map $F$ does not have shadowing. Fix $\epsilon=\tfrac14$ and let $\delta>0$ be arbitrary. The sequence $x_i=\min\{i\delta/2,\,1\}$ is a $\delta$-pseudo-orbit, since $d\big(x_{i+1},F(x_i)\big)=d\big(x_{i+1},[0,x_i]\big)=\max\{x_{i+1}-x_i,0\}\leq\delta/2<\delta$. Any orbit $(z_i)$ of $F$ is non-increasing, so $z_i\leq z_0$ for all $i$; if it $\epsilon$-shadowed $(x_i)$ then $z_0<x_0+\epsilon=\tfrac14$, hence $z_i<\tfrac14$ for all $i$, whereas $x_i\to1$, so $d(z_i,x_i)>\tfrac12$ for large $i$, a contradiction. Consequently, by Theorem \ref{thm:equiv}, $\underleftarrow{F}$ has no shadowing either, and by Corollary \ref{cor:all} neither do the shifts on $\underleftarrow{\lim} \underleftarrow{F}$ and on $\underleftarrow{\lim} F$. (The failure for $\underleftarrow{F}$ can also be seen directly through the conjugacy $\underleftarrow{F}=h\circ F\circ h^{-1}$, $h(x)=1-x$, or by repeating the drift argument downward.)
\end{example}

The hypotheses of Theorem \ref{thm:equiv} are also met by genuinely set-valued maps with rich dynamics, not merely by the elementary maps above.

\begin{example}\label{ex:doubling}
	Let $X=[0,1]$ and let $T:X\rightarrow 2^X$ be the set-valued doubling map whose graph is the union of the two closed segments
	$$\operatorname{Gr}(T)=\{(x,2x):0\leq x\leq\tfrac12\}\cup\{(x,2x-1):\tfrac12\leq x\leq1\},$$
	that is, $T(x)=\{2x\}$ for $x<\tfrac12$, $T(x)=\{2x-1\}$ for $x>\tfrac12$, and $T(\tfrac12)=\{0,1\}$. Since $X$ is compact and $\operatorname{Gr}(T)$ is closed (a finite union of segments), $T$ is upper semi-continuous, and it is clearly onto. Its inverse is the two-valued contraction
	$$F:=\underleftarrow{T},\qquad F(x)=\Big\{\tfrac{x}{2},\ \tfrac{x+1}{2}\Big\},$$
	a Mahavier map whose graph is the union of the two segments of slope $\tfrac12$ from $(0,0)$ to $(1,\tfrac12)$ and from $(0,\tfrac12)$ to $(1,1)$; it is Hausdorff continuous, onto, and properly set-valued, since every value $F(x)$ consists of two distinct points. The doubling map $T$ is open and expanding, hence has the shadowing property; therefore, by Theorem \ref{thm:equiv}, the multi-valued map $F$ has shadowing as well, and by Corollary \ref{cor:all} so do the shifts on $\underleftarrow{\lim} F$ and on $\underleftarrow{\lim} \underleftarrow{F}=\underleftarrow{\lim} T$. Here $\underleftarrow{\lim} F$ is a non-degenerate continuum (the inverse limit of the doubling map), so a topologically rich system carries the shadowing shift; this lies well beyond the elementary Examples \ref{ex:usc} and \ref{ex:nodrift}.
\end{example}

\begin{remark}\label{rem:hypotheses}
	The hypotheses of Theorem \ref{thm:equiv} are economical. Surjectivity $F(X)=X$ cannot be dropped: it is exactly what guarantees that $\underleftarrow{F}(x)=\{y:x\in F(y)\}$ is nonempty for every $x$, so that $\underleftarrow{F}$ is again a set-valued map $X\rightarrow 2^X$ (and that $\underleftarrow{\lim} F\neq\emptyset$); without it the statement loses its meaning. Upper semi-continuity enters only through the resulting closedness of the graph of $F$ on the compact space $X$, which is what keeps the diagonal limit in the proof an orbit; it is also the weakest regularity under which $F$ is guaranteed to take values in $2^X$. By contrast, neither lower semi-continuity (equivalently, continuity) nor openness of $F$ is required, in contrast with \cite{KhanKumarDas}, and Example \ref{ex:usc} shows that this gain is genuine. Finally, the normalization $\operatorname{diam}X=1$, used in Corollaries \ref{cor1}, \ref{cor2} and \ref{cor:all} through the metric $\rho$, is harmless: rescaling the metric of any compact metric space makes its diameter equal to $1$ without affecting any of the dynamical notions considered here.
\end{remark}
\section{Transitivity and mixing of a set-valued map and its inverse limit}\label{sectran}
First, we present a generalization. Fakhari and Ghane\cite{Fakhari2010} proved that for a continuous self-map $f$ on $X$, if $f$ has the shadowing property then $f$ is chain mixing if and only if it is topologically mixing. In fact, analogous results also hold for set-valued maps. Furthermore, there have been many publications on inverse limits of set-valued maps, but the associated dynamics of the set-valued maps and the relationships between a set-valued map and its inverse limit have not been studied in depth. Here we mainly study transitivity, weakly mixing, mixing, chain transitivity and chain mixing of a set-valued map and its inverse limit. These results parallel and extend the transitivity, sensitivity and chaos analysis of \cite{zhao2023}; we include them for completeness.

\begin{lemma}\label{thm6}
	Let $X$ be a compact metric space with metric $d$ and $F:X\rightarrow 2^X$ be an upper semi-continuous set-valued map. Let $F$ have shadowing. Then the following statements are equivalent:
	\begin{enumerate}
		
		\item $F$ is topologically mixing.
		
		\item $F$ is chain mixing.
	\end{enumerate}
\end{lemma}
\begin{proof}

	{$2\Rightarrow1$}: Suppose that $F$ has shadowing and $F$ is chain mixing. Let $U,V\subset X$ be two nonempty open sets. Select $x\in U$, $y\in V$ and $\epsilon>0$ such that $B(x, \epsilon)\subset U$ and $B(y, \epsilon)\subset V$. Let $\delta>0$ witness shadowing {of $F$} for $\epsilon$. Since $F$ is chain mixing, there is a positive integer $N$ such that for any integer $n\geq N$ there is a $\delta$-chain from $x$ to $y$ with length $n$. Denote the $\delta$-chain by $(x_i)_{i=1}^n$ with $x_1=x$ and $x_n=y$. By shadowing of $F$, there is $z_1\in X$ with an finite orbit $(z_i)_{i=1}^n$ such that $d(z_i, x_i)<\epsilon$ for any $1\leq i\leq n$. So $z_1\in B(x,\epsilon)\subset U$ implies $z_1\in U$ and $z_n\in B(y, \epsilon)\subset V$ implies $z_n\in V$. So, $F$ is topologically mixing.
	
	$2\Leftarrow1$: Suppose that $F$ has shadowing and $F$ is topologically mixing. Let $x,y\in X$ and $\delta>0$. Since $B(F(x),\delta)$ {and} $B(y, \delta)\subset X$ are two non-empty open sets, there is $N\in\mathbb{Z}^+$ such that for any integer $n\geq N$, there is $z_1\in B(F(x), \delta)$ with an orbit $(z_i)_{i=1}^\infty$ with $z_n\in B(y, \delta)$. Put $z_0'=x$, $z_i'=z_i$, $1\leq i\leq n-1$, $z_n'=y$. Then, $d(F(z_0'),z_1)<\delta$ {and} $d(F(z_i'),z_{i+1}')<\delta$ for any $1\leq i\leq n-2$. By $z_n\in B(y,\delta)$ and $z_n\in F(z_{n-1})$,  $d(F(z_{n-1}',z_n')=d(F(z_{n-1}),y)<\delta$. Hence, there exists a $\delta$-chain from $x$ to $y$ with length $n$. So, $F$ is chain mixing.
	
\end{proof}

Based on Theorem 3.13 from \cite{Luo2020} and the above Lemma \ref{thm6}, the following conclusion can be drawn.
\begin{corollary}
	Let $X$ be a compact metric space with metric $d$ and $F:X\rightarrow 2^X$ be an upper semi-continuous set-valued map. Let $F$ have shadowing. Then the following statements are equivalent:
	\begin{enumerate}
		\item $F$ is totally transitive.
		
		\item $F$ is weakly mixing.
		
		\item $F$ is topologically mixing.
		
		\item $F$ has the specification property.
		
		\item $F$ is chain mixing.
	\end{enumerate}
\end{corollary}

\begin{proposition}\label{prop1}
	Suppose that $X$ is a compact metric space with metric $d$ and $F:X\rightarrow 2^X$ is an upper semi-continuous set-valued map. Let $F(X)=X$. $F$ is transitive (resp. {weakly mixing, mixing, chain transitive, chain mixing}) if and only if $\underleftarrow{F}$ is transitive (resp. {weakly mixing, mixing, chain transitive, chain mixing}).
\end{proposition}
\begin{proof}
	{Based on $\underleftarrow{(\underleftarrow{F})}=F$, to prove this proposition, it is only necessary to prove necessity, that is, if $F$ is transitive (resp. {weakly mixing, mixing, chain transitive, chain mixing}) then $\underleftarrow{F}$ is transitive (resp. {weakly mixing, mixing, chain transitive, chain mixing}).}
	
	(1)Transitive: 
	
	{Let $U,V\subset X$ be two nonempty open sets. Since $F$ is transitive, there is $n\in\mathbb{N}$ such that $F^n(U)\bigcap V\neq\emptyset$. Then there is $x\in U$ such that one can find  $x_n\in F^n(x)\bigcap V$. Furthermore, there is $x\in\underleftarrow{F}^n(x_n)\bigcap U$ for some $x_n\in V$. That is, $\underleftarrow{F}^n(V)\bigcap U\neq\emptyset$. So, $\underleftarrow{F}$ is transitive.}
	
	(2)Weakly mixing:
	
	{Let $U,V\subset X$ be two nonempty open sets. Since $F$ is weakly mixing,  there is a strictly increasing subsequence $\{n_i\}$ of $\mathbb{N}$ such that for any $i\geq 1$ and any $0\leq j\leq i$, $F^{n_i+j}(U)\bigcap V\neq\emptyset$.
		Then,  $\underleftarrow{F}^{n_i+j}(V)\bigcap U\neq\emptyset$. So, $\underleftarrow{F}$ is weakly mixing.}
	
	(3)Mixing:
	Let $U,V\subset X$ be two nonempty open sets. Since $F$ is mixing, there is $M\in\mathbb{N}$ such that for any $n\geq M$, $F^{n}(U)\bigcap V\neq\emptyset$. Then, $\underleftarrow{F}^{n}(V)\bigcap U\neq\emptyset$. So, $\underleftarrow{F}$ is mixing.
	
	(4)Chain transitive:
	
	Let $x,y\in X$ and $\delta>0$. Since $\underleftarrow{F}$ is upper-continuous and $X$ is compact with the metric $d$, there is $\delta'>0$ such that for any $u\in X$, 
	\begin{equation}\label{ex1}
		\underleftarrow{F}(t)\subset B(\underleftarrow{F}(u),\delta), \forall t\in B(u,\delta').
	\end{equation}
	Since $F$ is chain transitive, there is a $\delta'$-chain $(x_i)^0_{i=n}$($n$ is a positive integer) of $F$ from $y$ to $x$ with $x_n=y$ and $x_0=x$. Then for any $0\leq i<n$, $d(F(x_{i+1}),x_{i})<\delta'$. That is, $F(x_{i+1})\bigcap B(x_i,\delta')\neq\emptyset$. Then $x_{i+1}\in\underleftarrow{F}(B(x_i,\delta'))$. By {Equation (\ref{ex1})}, $d(x_{i+1},\underleftarrow{F}(x_i))<\delta$. Then $(x_i)^n_{i=0}$ is a $\delta$-chain of $\underleftarrow{F}$ from $x$ to $y$. So, $\underleftarrow{F}$ is chain transitive.
	
	(5)Chain mixing:
	
	{Let $x,y\in X$ and $\delta>0$. Since $\underleftarrow{F}$ is upper-continuous and $X$ is compact with the metric $d$, there is $\delta'>0$ such that for any $u\in X$, $\underleftarrow{F}(t)\subset B(\underleftarrow{F}(u),\delta), \forall t\in B(u,\delta')$.
		Since $F$ is chain mixing, there is $N>0$ such that for any $n\geq N$ there exists a $\delta'$-chain $(x_i)^0_{i=n}$($n$ is a positive integer) of $F$ from $y$ to $x$. Then $(x_i)^n_{i=0}$ is a $\delta$-chain of $\underleftarrow{F}$ from $x$ to $y$. So, $\underleftarrow{F}$ is chain mixing.}
\end{proof}

\begin{theorem}\label{them4}
	Suppose that $X$ is a compact metric space with metric $d$ and $F:X\rightarrow 2^X$ is an upper semi-continuous set-valued map. Let $F(X)=X$. {If $\sigma$ is transitive(resp. weakly mixing, mixing) on $\underleftarrow{\lim} F$, then $F$ is transitive(resp. weakly mixing, mixing).}
\end{theorem}
\begin{proof}
	
	{(1)$\sigma$ is transitive on $\underleftarrow{\lim} F$ $\Rightarrow$ $F$ is transitive:}
	
	Let $U_0$, $V_0\subset X$ be two non-empty open sets. {By $F(X)=X$,} we can get two non-empty open sets {$\mathcal{U}=U_0\times X\times X\times\cdots$, $\mathcal{V}=V_0\times X\times X\times\cdots\subset\underleftarrow{\lim} F $.} By the transitivity of {$\sigma$ on $\underleftarrow{\lim} F$}, there exists $n\in\mathbb{N}$ such that $\sigma^n(\mathcal{U})\cap\mathcal{V}\neq\emptyset$. Then, one can find $(x_i)_{i=0}^\infty\in\mathcal{U}$ satisfying $(x_n,x_{n+1},\cdots)\in\mathcal{V}$, in which $x_0\in F^n(x_n)$. Construct an orbit $(y_i)_{i=0}^\infty$ {of} $F$ with $y_n=x_0$ and $y_0=x_n$. Then, $y_0\in V_0$ and $y_n\in U_0$. So, $F$ is transitive.
	
	{(2)$\sigma$ is weakly mixing on $\underleftarrow{\lim} F$ $\Rightarrow$ $F$ is weakly mixing:}
	
	Let $U_0$, $V_0\subset X$ be two non-empty open sets. {By $F(X)=X$,} we can get two non-empty open sets {$\mathcal{U}=U_0\times X\times X\times\cdots$, $\mathcal{V}=V_0\times X\times X\times\cdots\subset\underleftarrow{\lim} F $.} By the weakly mixing of {$\sigma$ on $\underleftarrow{\lim} F$}, there is a strictly increasing subsequence $\{n_m\}$ of $\mathbb{N}$ such that $$\bigcup_{m=1}^\infty\{n_m,n_m+1,\cdots,n_m+m\}\subset N(\mathcal{U},\mathcal{V}).$$ Let $n\in \bigcup_{m=1}^\infty\{n_m,n_m+1,\cdots,n_m+m\}$.Then one can find $(x_i)_{i=0}^\infty\in\mathcal{U}$ satisfying $(x_n,x_{n+1},\cdots)\in\mathcal{V}$, in which $x_0\in F^n(x_n)$. Construct an orbit $(y_i)_{i=0}^\infty$ for $F$ with $y_n=x_0$ and $y_0=x_n$. Then, $y_0\in V_0$ and $y_n\in U_0$, that is, $F^n(U)\bigcap V\neq\emptyset$. By $n\in \bigcup_{m=1}^\infty\{n_m,n_m+1,\cdots,n_m+m\}$, $F$ is weakly mixing.
	
	{(3)$\sigma$ is mixing on $\underleftarrow{\lim} F$ $\Rightarrow$ $F$ is mixing:}
	
	Let $U_0$, $V_0\subset X$ be two non-empty open sets. {By $F(X)=X$,} we can get two non-empty open sets {$\mathcal{U}=U_0\times X\times X\times\cdots$, $\mathcal{V}=V_0\times X\times X\times\cdots\subset\underleftarrow{\lim} F $.} By the mixing of {$\sigma$ on $\underleftarrow{\lim} F$}, there is a non-negative integer $M$ such that for any $n\geq M$, $\sigma^n(\mathcal{U})\cap\mathcal{V}\neq\emptyset$. Let $n\geq M$. Then one can find $(x_i)_{i=0}^\infty\in\mathcal{U}$ satisfying $(x_n,x_{n+1},\cdots)\in\mathcal{V}$, in which $x_0\in F^n(x_n)$. Construct an orbit $(y_i)_{i=0}^\infty$ for $F$ with $y_n=x_0$ and $y_0=x_n$. Then, $y_0\in V_0$ and $y_n\in U_0$, that is, $F^n(U)\bigcap V\neq\emptyset$. By $n\geq M$, $F$ is mixing.
	\end{proof}
	\begin{theorem}\label{them5}
Suppose that $X$ is a compact metric space with metric $d$ and $F:X\rightarrow 2^X$ is an upper semi-continuous set-valued map. Let $F(X)=X$. If $\sigma$ is chain mixing(resp. chain transitive) on $\underleftarrow{\lim} F$, then $F$ is chain mixing(resp. chain transitive).
\end{theorem}
\begin{proof}    

	(1)Chain transitive:
	
	Let $x_0, y_0$ be two points of $X$ and $\delta>0$. Then we can get two points $(x_i)_{i=0}^\infty$, $(y_i)_{i=0}^\infty\in\underleftarrow{\lim} F$ with $x_{i}\in F(x_{i+1})$ and $y_{i}\in F(y_{i+1})$ for all $i\in\mathbb{N}$. By the upper semi-continuity of $F$, there exists $0<\delta_1<\frac{\delta}{2}$ such that $d(u,x)<\delta_1$ implies $d(u', F(x))<\frac{\delta}{2}$, $\forall u'\in F(u)$. For $x$, $y\in\underleftarrow{\lim} F$ and $\delta_1>0$, by the chain transitivity of $\underleftarrow{\lim} F$, there is a $\delta_1$-chain $z^0=x$, $z^1,z^2,\cdots z^{n-1}$, $z^n=y$($n$ is a positive integer) from $x$ to $y$ with length $n+1$, in which $z^i=(z_j^i)_{j=0}^\infty$ and $z^i_j\in F(z^i_{j+1})$ for any $0\leq i\leq n$, $j\geq 0$. Then, for any $0\leq i\leq n-1$, $\rho(z^{i+1}, \sigma(z^i))<\delta_1$. Hence, for any $0\leq i\leq n-1$, $d(z^{i+1}_0,z^i_1)<\delta_1$. Then, $d(F(z^{i+1}_{0}),z^{i}_0)<\frac{\delta}{2}$, for any $0\leq i\leq n-1$.
	
	Put $u_0=z^n_0=y_0$, $u_1=z^{n-1}_0$, $u_2=z^{n-2}_0$, $\cdots$, $u_{n-1}=z^{1}_0$, $u_n=z^0_0=x_0$. 
	Then, $(u_i)_{i=0}^{n}$ is a $\delta$-chain of $F$ from $y_0$ to $x_0$ with length $n+1$. So, $F$ is chain transitive.
	
	(2)Chain mixing:
	
	Let $x_0, y_0$ be two points of $X$ and $\delta>0$. Then we can get two points $(x_i)_{i=0}^\infty$, $(y_i)_{i=0}^\infty\in\underleftarrow{\lim} F$ with $x_{i}\in F(x_{i+1})$ and $y_{i}\in F(y_{i+1})$ for all $i\in\mathbb{N}$. By the upper semi-continuity of $F$, there exists $0<\delta_1<\frac{\delta}{2}$ such that $d(u,x)<\delta_1$ implies $d(u', F(x))<\frac{\delta}{2}$, $\forall u'\in F(u)$. For $x$, $y\in\underleftarrow{\lim} F$ and $\delta_1>0$, by the chain mixing of $\underleftarrow{\lim} F$, there is a positive integer $N$ such that for any integer $n\geq N$, there is a $\delta_1$-chain $z^0=x$, $z^1,z^2,\cdots z^{n-1}$, $z^n=y$ from $x$ to $y$ with length $n+1$, in which $z^i=(z_j^i)_{j=0}^\infty$ and $z^i_j\in F(z^i_{j+1})$ for any $0\leq i\leq n$, $j\geq 0$. Then, for any $0\leq i\leq n-1$, $\rho(z^{i+1}, \sigma(z^i))<\delta_1$. Hence, for any $0\leq i\leq n-1$, $d(z^{i+1}_0,z^i_1)<\delta_1$. So, $d(F(z^{i+1}_{0}),z^{i}_0)<\frac{\delta}{2}$, for any $0\leq i\leq n-1$.
	
	Put $u_0=z^n_0=y_0$, $u_1=z^{n-1}_0$, $u_2=z^{n-2}_0$, $\cdots$, $u_{n-1}=z^{1}_0$, $u_n=z^0_0=x_0$. 
	Then, $(u_i)_{i=0}^{n}$ is a $\delta$-chain of $F$ from $y_0$ to $x_0$ with length $n+1$. So, $F$ is chain mixing.
	
	\end{proof}

	\section{Conclusions}\label{seccon}
	We study the properties of shadowing,  transitivity, weakly mixing, mixing, chain transitivity and chain mixing for the set-valued map and its inverse limit. We show that
	\begin{itemize}
\item[1] Let $F(X)=X$ and $\operatorname{diam}X=1$. Then $F$ has shadowing $\Leftrightarrow$ $\sigma$ has shadowing on $\underleftarrow{\lim} \underleftarrow{F}$.

\item[2] Let $F(X)=X$ and $\operatorname{diam}X=1$. Then $\underleftarrow{F}$ has shadowing $\Leftrightarrow$ $\sigma$ has shadowing on $\underleftarrow{\lim} F$.

\item[3] Let $F(X)=X$ and $\operatorname{diam}X=1$. Then $F$ has shadowing $\Leftrightarrow$ $\underleftarrow{F}$ has shadowing; consequently $F$, $\underleftarrow{F}$, $(\underleftarrow{\lim} \underleftarrow{F},\sigma)$ and $(\underleftarrow{\lim} F,\sigma)$ all have shadowing simultaneously. 
\item[4] Let $F$ have shadowing. Then $F$ is totally transitive $\Leftrightarrow$ $F$ is weakly mixing $\Leftrightarrow$ $F$ is mixing $\Leftrightarrow$ $F$ has specification property $\Leftrightarrow$ $F$ is chain mixing.
\item [5] Let $F(X)=X$. Then $F$ is transitive (resp. chain transitive, chain mixing, weakly mixing, mixing) $\Leftrightarrow$ $\underleftarrow{F}$ is transitive (resp. chain transitive, chain mixing, weakly mixing, mixing).

\item[6] Let $F(X)=X$. If $\sigma$ is transitive(resp. weakly mixing, mixing, chain transitive, chain mixing) on $\underleftarrow{\lim} F$, then $F$ is transitive(resp. weakly mixing, mixing, chain transitive, chain mixing).

\end{itemize}

Conclusion 1 strengthens the earlier one-sided implication (Theorem \ref{thm2}) to an equivalence, characterizing the shadowing of $F$ by that of the shift on $\underleftarrow{\lim} \underleftarrow{F}$, and Conclusion 2 is its dual for $\underleftarrow{F}$ and $\underleftarrow{\lim} F$. Conclusion 3 then closes the loop: the shadowing of $F$ and that of $\underleftarrow{F}$ are always equivalent, with no continuity or openness assumption, which strengthens the recent theorem of Khan, Kumar and Das \cite{KhanKumarDas} obtained under such assumptions (see Theorem \ref{thm:equiv} and Remark \ref{rem:strengthen}). The key point is that one reverses the witnesses of a pseudo-orbit rather than the pseudo-orbit itself, so that only the upper semi-continuity of $F$ is used; Example \ref{ex:usc} exhibits a genuinely discontinuous, non-open such map to which the result applies but the earlier theorem does not. Examples \ref{ex:nodrift} and \ref{ex:doubling} further illustrate the equivalence in its negative form and on a non-degenerate, dynamically rich map, while Remark \ref{rem:hypotheses} shows that surjectivity and upper semi-continuity are essential whereas continuity, openness and the diameter normalization are not.

The above conclusions generalize the corresponding results for single-valued continuous self-maps on a metric
space and their inverse limits. This not only helps us
to study the set-valued maps more deeply, but also makes it available to use the
relatively simple set-valued maps to understand the relatively complex generalized
inverse limits.

\end{document}